\documentclass[a4paper,11pt]{amsart}
\usepackage{graphicx, amsfonts, amsmath, pstricks,pst-plot,subfig}
\usepackage{caption}
\usepackage{amsthm}
\usepackage{url}
\newcommand{\rad}[1]{\text{rad}(#1)}
\newtheorem{theorem}{Theorem}
\newtheorem{lemma}{Lemma}
\newtheorem{conjecture}{Conjecture}
\newtheorem*{korselt}{Korselt's Criterion}
\title{RADICALLY WEAKENING THE LEHMER AND CARMICHAEL CONDITIONS}
\author{NATHAN MCNEW}
\address{Department of Mathematics, Dartmouth College \\ Hanover, NH 03755 USA}
\email{nathan.g.mcnew@dartmouth.edu}
\date{}
   
\begin{document}
\begin{abstract}
Lehmer's totient problem asks if there exist composite integers $n$ satisfying the condition $\varphi(n)|(n-1)$, (where $\varphi$ is the Euler-phi function) while Carmichael numbers satisfy the weaker condition $\lambda(n)|(n-1)$ (where $\lambda$ is the Carmichael universal exponent function). We weaken the condition further, looking at those composite $n$ where each prime divisor of $\varphi(n)$ also divides $n-1$. (So $\rad{\varphi(n)}|(n-1)$.) While these numbers appear to be far more numerous than the Carmichael
numbers, we show that their distribution has the same rough upper bound as that of the Carmichael numbers, a bound which is heuristically tight.
\end{abstract}
  
\maketitle
  
  
\section{Introduction}
Let $\varphi(n)$ denote the Euler totient function of $n$. Lehmer \cite{Lehmer} asked whether there exist composite positive integers $n$ such that $\varphi(n)|n-1$.  Integers which satisfy this \lq\lq Lehmer Condition" are sometimes referred to as Lehmer numbers, however no examples are known.  Cohen and Hagis \cite{Cohen} have shown that any Lehmer numbers would necessarily have at least 14 prime factors, and computations by Pinch \cite{Pinch_Lehmer} show that any examples must be greater than $10^{30}$.  Further, Luca and Pomerance \cite{Luca} have shown that if $\mathcal{L}(x)$ is the number of Lehmer numbers up to $x$ then, as $x \to \infty$,
\[\mathcal{L}(x) \leq \frac{x^{1/2}}{(\log x)^{1/2+o(1)}}.\]
    
Carmichael numbers are the composite integers $n$ which satisfy the congruence $a^n \equiv a \pmod{n}$ for every integer $a$.  (Fermat's little theorem guarantees that any prime number $n$ satisfies this congruence.)   Carmichael numbers were first characterized by Korselt \cite{Korselt} in 1899:
  
\begin{korselt}A composite number $n$ is a Carmichael number if and only if $n$ is square-free, and for each prime $p$ which divides $n$, $p - 1$ divides $n - 1$.
\end{korselt}
  
Korselt did not find any Carmichael numbers, however.  The smallest, 561, was found by Carmichael in 1910 \cite{Carmichael}.  Carmichael also gave a new characterization of these numbers as those composite $n$ which satisfy $\lambda(n)|n-1$, where $\lambda(n)$, the Carmichael lambda function, denotes the size of the largest cyclic subgroup of $(\mathbb{Z}/n\mathbb{Z})^\times$.
  
Since $\lambda(n)|\varphi(n)$ for every integer $n$, the Carmichael property can be viewed as a weakening of the Lehmer property.  Every Lehmer number would also be a Carmichael number.  In contrast to the Lehmer numbers, it is known, due to Alford, Granville and Pomerance \cite{Infinite}, that there are infinitely many Carmichael numbers.  Pomerance \cite{Pomerance} also proves an upper bound for the number $C(x)$ of Carmichael numbers up to $x$, namely as $x \to \infty$,
\begin{equation}C(x) \leq x^{1-\{1+o(1)\}\log\log\log x /\log\log x}, \label{car}
\end{equation}
and presents a heuristic argument that this is the true size of $C(x)$.
  
Grau and Oller-Marc{\'e}n \cite{Grau} present other possible weakenings of the Lehmer property: looking at the sets of those $n$ such that $\varphi(n)|(n-1)^k$ for a fixed value of $k$ as well as the set of those $n$ for which $\varphi(n)|(n-1)^k$ for some $k$, that is all of the primes dividing $\varphi(n)$ also divide $n-1$.  Note that this last set is a weakening of both the Lehmer and Carmichael properties, since $\lambda(n)$ and $\varphi(n)$ have the same prime divisors.  Our results resolve several conjectures that Grau and Oller-Marc{\'e}n made in their paper.  
  
We focus primarily on this final set.  Let $\kappa(n) = \rad{\varphi(n)}$ denote the product of the primes which divide the value $\varphi(n)$. (Note that $\kappa(n) = \rad{\varphi(n)} = \rad{\lambda(n)}$.)  Let $\mathbb{K}(x) $ be the set of composite numbers $n \leq x$ which satisfy $\kappa(n)| n-1$, and let $K(x) = |\mathbb{K}(x)|$. (Observe that every prime number $p$ trivially satisfies $\kappa(p)|p-1$.)
  
We prove that the upper bound \eqref{car} for $C(x)$ also applies for $K(x)$.  We also present upper bounds for the number of $n \in \mathbb{K}(x)$ which are the product of a fixed number of primes, as well as several related conjectures and computations.
  
\section{The Upper Bound}
The condition for $n$ to be a member of $\mathbb{K}(x)$ is substantially weaker than that required for $n$ to be a Carmichael number, and computations (see Section \ref{sec:computations}) show that $K(x)$ appears to be substantially greater than $C(x)$.  It is therefore somewhat surprising to find that these two functions have the same rough upper bound.  Our proof of this fact is similar to the one for $C(x)$ in \cite{Pomerance}.
\begin{theorem} Define $L(x) = \exp(\log x \frac{\log\log\log x}{\log\log x})$. Then as $x \rightarrow \infty$, \label{Main}
\[K(x) \leq \frac{x}{L(x)^{1 + o(1)}}.\]
\begin{proof} We consider first those integers $n \leq x$ which have a large prime divisor.  Specifically, let $P(n)$ denote the largest prime divisor of $n$, and write $n = mp$ where $p = P(n)$.  We restrict our attention to those $n$ with $P(n) > L(x)^2$, and let $K'(x) = \#\{n \in \mathbb{K}(x) \mid P(n) > L(x)^2\}$.
  
If $n=mp$ is to satisfy $\kappa(n) | n-1$, then we must have $m \leq \frac{x}{p} $, and $m$ must be congruent to 1 $\pmod{ \rad{p-1}}$. Thus, for any fixed $p$ there are at most $1 + \lfloor\frac{x}{p\cdot\rad{p-1}}\rfloor$ possibilities for $m$.  Requiring $n$ to be composite (thus $m \neq 1$) leaves us with at most $\frac{x}{p\cdot\rad{p-1}}$ possibilities.
  
Thus we see that 
\begin{align} K'(x) &= \sum_{\substack{n = mp\leq x \\ p > L(x)^2 \\ \kappa(n)|n-1} } 1 \leq \sum_{\substack{p > L(x)^2} } \frac{x}{p \hspace{1mm} \rad{p-1} } \notag \\
&\leq \sum_{\substack{p > L(x)^2} } \frac{x}{(p-1) \hspace{1mm} \rad{p-1}}. \label{squarefree}
\end{align}
  
Now, we observe that for each prime $p$, the denominator in \eqref{squarefree} is a squarefull number, and that any squarefull number can be represented uniquely as $d \hspace{1mm} \rad{d}$ for some integer $d$.  We can therefore replace this sum with a sum over all squarefull numbers:
  
\begin{align*} 
\sum_{\substack{p > L(x)^2} } \frac{x}{(p-1) \hspace{1mm} \rad{p-1} } & \leq \sum_{\substack{d > L(x)^2 \\ d \text{ squareful}} } \frac{x}{d} .
\end{align*}
Using partial summation and the fact that 
\[\sum_{\substack{n \leq x \\ n \text{ squareful}} } 1 = \frac{\zeta(3/2)}{\zeta(3)} x^{1/2} + O(x^{1/3}),\]
we see that 
\[K'(x) \leq  \sum_{\substack{d > L(x)^2 \\ d \text{ squareful}} } \frac{x}{d}  \ll \frac{x}{L(x)}.\]
  
We may assume that $n>\frac{x}{L(x)}$, so to prove the theorem, it suffices to count those $n$ with $\frac{x}{L(x)} < n \leq x$ and $P(n) \leq L(x)^2$.  We denote this count by $K''(x)$. Observe that every such $n$ has a divisor $d$ satisfying \begin{equation}\label{div} \frac{x}{L(X)^3} < d \leq \frac{x}{L(x)}.\end{equation}  
Write $n =md$, so $m\leq \frac{x}{d}$. Now, if $n = md$ is to satisfy $\kappa(md)|md-1$, we have $m \equiv 1 \pmod{\kappa(d)}$, and since $(n,\kappa(n)) = 1$ and $\kappa(d)|\kappa(n)$ we know $(d,\kappa(d))=1$. Thus the Chinese remainder theorem implies that there are at most $1 + \lfloor\frac{x}{d\kappa(d)}\rfloor$ possibilities for $m$. Thus 
\[K''(x) \leq \sideset{}{'}\sum \left( 1 + \frac{x}{d\kappa(d)} \right) \leq \frac{x}{L(x)} +  \sideset{}{'}\sum \left\lfloor \frac{x}{d\kappa(d)} \right\rfloor ,\]
where $\sideset{}{'}\sum$ denotes a sum over $d$ satisfying \eqref{div}.  If $d\kappa(d) \leq x$ and $d$ satisfies \eqref{div}, then $\kappa(d) < L(x)^3$, so that
\begin{align}K''(x) &\leq \frac{x}{L(x)} +  \sideset{}{'}\sum \left\lfloor \frac{x}{d\kappa(d)} \right\rfloor \notag \\
&\leq \frac{x}{L(x)} + x \sum_{c \leq L(x)^3} \frac{1}{c} \sideset{}{'}\sum_{\kappa(d) = c}\frac{1}{d}. \label{doublesum}
\end{align}
  
We treat the inner sum in \eqref{doublesum} by partial summation: 
\begin{align} \sideset{}{'}\sum_{\kappa(d) = c}\frac{1}{d} = \frac{L(x)}{x} \sideset{}{'}\sum_{\kappa(d) = c} 1 + \int_{\frac{x}{L(x)^3}}^{\frac{x}{L(x)}}\hspace{2mm}\frac{1}{t^2} \hspace{2mm} \sideset{}{'}\sum_{\substack{\kappa(d) = c \\ d< t}} 1\hspace{2mm} dt. \label{parsum}
\end{align}
We are thus interested in obtaining an upper bound for $\mathcal{K}(t,c)$, the number of $d \leq t$ with $\kappa(d) = c$.
  
\begin{lemma} \label{uniform} As $t \rightarrow \infty$, $\mathcal{K}(t,c) \leq \frac{t}{L(t)^{1+o(1)}}$ uniformly for all $c$.  
\end{lemma}
Before proving the lemma, we see that using this upper bound in \eqref{parsum} gives us 
\begin{align*} \sideset{}{'}\sum_{\kappa(d) = c}\frac{1}{d} &\leq \frac{L(x)}{x} \mathcal{K}(\tfrac{x}{L(x)},c) + \int_{\frac{x}{L(x)^3}}^{\frac{x}{L(x)}}\hspace{2mm}\frac{1}{t^2} \hspace{2mm} \mathcal{K}(t,c)\hspace{2mm}  dt \\
&\leq {L(\tfrac{x}{L(x)})^{-1+o(1)}} + \int_{\frac{x}{L(x)^3}}^{\frac{x}{L(x)}}\hspace{2mm}\frac{1}{tL(t)^{1+o(1)}} \hspace{2mm} dt \\
&= L(x)^{-1+o(1)}
\end{align*}
as $x \to \infty$. This can be used in \eqref{doublesum} to see that $K''(x) \leq  \frac{x}{L(x)^{1+o(1)}}$.  The theorem then follows immediately from our estimates of $K'(x)$ and $K''(x)$.
  
It thus remains to prove Lemma \ref{uniform}.  We may assume that $c \leq t$, otherwise $\mathcal{K}(t,c) = 0$.  Then, for any $r > 0$ we can write:
\begin{align*}
\mathcal{K}(t,c) &= \sum_{\substack{d \leq t\\ \kappa(d) = c}} 1 \leq t^r \sum_{\substack{\kappa(d) = c}} d^{-r}\\
& \leq t^r \sum_{\substack{p|d \Rightarrow \rad{p-1}|c}} d^{-r} = t^r \prod_{\substack{\rad{p-1}|c}}\tfrac{1}{1-p^{-r}}.
\end{align*}
Assuming $r \geq 1/2 + \epsilon$ then 
\begin{align*} \prod_{\substack{\rad{p-1}|c}}\tfrac{1}{1-p^{-r}} &= \exp\left(\sum_{\substack{\rad{p-1}|c}}-\log({1-p^{-r})}\right) = \exp\left(\sum_{\substack{\rad{p-1}|c}}\hspace{3mm}\sum_{n=1}^\infty \frac{p^{-nr}}{n}\right) \\
&= \exp\left(\left(\sum_{\substack{\rad{p-1}|c}}p^{-r}\right) + 
O_\epsilon(1)\right).
\end{align*}
So we have
\begin{align*} \mathcal{K}(t,c) &\ll_{\epsilon} t^r \exp \left(\sum_{\text{rad}(p-1) | c} p^{-r} \right) \leq t^r \exp \left(\sum_{\text{rad}(l) | c} l^{-r} \right) \\
& = t^r \exp \left(\prod_{p|c} (1-p^{-r})^{-1} \right) \leq t^r \exp\exp \left(\sum_{p|c}p^{-r} + O_\epsilon(1) \right)
 \end{align*}
by applying this trick a second time.
Now, $\sum_{p|c}p^{-r}$ is maximized when $c$ is the largest primorial up to $t$, in other words $c = p_1p_2 \cdots p_k < t$, where $p_i$ is the $i$th prime.  Further, if $t$ is sufficiently large, then the prime number theorem implies that $p_k \leq 2\log(t)$ and thus 
\[ \sum_{p|c}p^{-r} \leq \sum_{p<2\log(t)}p^{-r} \]
Choose $r = 1-(\log\log\log t)/(\log\log t)$.  Thus for large $t$, we may choose $\epsilon = 1/4$. Then we have $t^r = \frac{t}{L(t)}$ and \[\sum_{p<2\log(t)}p^{-r}  = O(\log\log t/\log\log\log t).\] Thus
\begin{align*}\mathcal{K}(t,c) &\leq t^r \exp\exp \left(\sum_{p|c}p^{-r} + O_\epsilon(1) \right) \\
&= \frac{t}{L(t)}\exp\exp(O(\log\log t /\log\log\log t)) = \frac{t}{L(t)^{1 + o(1)}},
\end{align*}
as $t \to \infty$, which completes the proof of the lemma.
\end{proof}
\end{theorem}
 
 \section{Bounds for integers in $\mathbb{K}(x)$ with $d$ prime factors}
 Since the integers satisfying our condition have a similar behavior to the Carmichael numbers assymptotically, it is natural to wonder if the behavior of those numbers with a fixed number of prime factors behaves similarly as well.  Granville and Pomerance \cite{Granville} conjecture that the number, $C_d(x)$, of Carmichael numbers with exactly $d$ prime factors is $x^{1/d + o(1)}$ when $d\geq 3$, and as $x \to \infty$.  This has not been proven for any $k$. However, Heath-Brown \cite{Heath} has shown that $C_3(x) \ll_\epsilon x^{7/20+\epsilon}$. Note that there are no Carmichael numbers with 2 prime factors.
  
 Let $K_d(x) = \#\{n \in \mathbb{K}(x), \omega(n)=d\}$ count the integers satisfying our condition up to $x$ with exactly $d$ prime factors.  Using the same method as the first part of Theorem \ref{Main} we can prove
 
 \begin{theorem} Uniformly for $d\geq 2$ we have the bound $K_d(x) \ll x^{1-\frac{1}{2d}}$. \label{kd}
 \end{theorem}
 \begin{proof} Consider first those $n > x/2$.  Since $n$ has $d$ prime factors, the largest prime factor must then satisfy $P(n) > (x/2)^{1/d}$.  Applying the same argument used for integers $n$ with a large prime factor in Theorem \ref{Main}, we find that the total contribution of such integers is at most $O(x^{1-\frac{1}{2d}})$. Hence, $K_d(x) -  K_d(x/2) \ll x^{1-\frac{1}{2d}}$.  
 
Now summing dyadically we have 
 \[K_d(x) = \sum_{i=0}^\infty K_d(2^{-i}x)-K_d(2^{-i-1}x) \ll \sum_{i \geq 0} \left(\frac{x}{2^i}\right)^{1-\frac{1}{2d}} \ll x^{1-\frac{1}{2d}}. \]
 \end{proof}
 
In contrast to the situation for Carmichael numbers, there do exist numbers satisfying our condition with two prime factors, and we can prove a substantially better bound than that of Theorem \ref{kd} in this case.  As a matter of fact, their behavior appears to be like that conjectured for Carmichael numbers with a given number of prime factors.
 
 \begin{theorem} The numbers in $\mathbb{K}(x)$ with exactly two prime factors satisfy the bound $K_2(x) \leq x^{1/2}\exp\left(\frac{2(2\log x)^{1/2}}{\log\log x}\left(1+O\left(\tfrac{1}{\log\log x}\right)\right)\right)$.
\end{theorem} 
\begin{proof}Write $n = pq \leq x$.  Since $\kappa(pq) = \rad{(p-1)(q-1)}$ and $pq-1 = (p-1)(q-1)+(p-1)+(q-1)$ we have that $\kappa(pq)|pq-1$ if and only if $\rad{p-1}=\rad{q-1}$. Thus
\begin{align*}
K_2(x) &= \sum_{\substack{pq\leq x\\ \kappa(pq)|pq-1}} 1
 \hspace{2mm}= \sum_{\substack{pq\leq x\\ \rad{p-1}=\rad{q-1}}} \hspace{-5mm} 1 \hspace{4mm} \leq \sum_{\substack{(m+1)(n+1) \leq x\\ \rad{m}=\rad{n}}} 1 \\ &\leq \sum_{\substack{mn \leq x\\ \rad{m}=\rad{n}}} 1 \hspace{2mm} \leq x^r \sum_{\substack{mn\leq x\\ \rad{m}=\rad{n}}} \frac{1}{(mn)^r} 
\end{align*}
for any $r \geq 0$.  We can rewrite this as a double sum: 
\begin{align*}
 x^r \hspace{-7mm} \sum_{\substack{mn\leq x\\ \rad{m}=\rad{n}}} \frac{1}{(mn)^r} &=x^r \sum_{m\leq x} \frac{1}{m^r}\sum_{\substack{n\leq x/m\\ p|m \text{ iff } p|n}} \frac{1}{n^r} \leq x^r \sum_{m\leq x} \frac{1}{m^r}\prod_{\substack{p|m}} \frac{\frac{1}{p^r}}{1-\frac{1}{p^r}} \\
 &= x^r \sum_{m\leq x} \frac{1}{m^r\rad{m}^r}\prod_{\substack{p|m}} \frac{1}{1-p^{-r}}\\
 &= x^r \sum_{m\leq x} \frac{1}{m^r\rad{m}^r}\exp\left(\sum_{\substack{p|m}}-\log\left(1-p^{-r}\right)\right)\\
 & = x^r\sum_{m\leq x} \frac{1}{m^r\rad{m}^r}\exp\left(\sum_{\substack{p|m}} \sum_{j=1}^\infty \frac{p^{-jr}}{j}\right).
\end{align*}
 As in the proof of Lemma \ref{uniform}, we can replace the condition $p|m$ above with $p\leq 2 \log x$, and $m\hspace{1mm}\rad{m}$ by a squareful integer $d$.  We also set $r = 1/2$. Thus:
 \begin{align*}
 x^{1/2}\sum_{m\leq x} &\frac{1}{m^{1/2}\rad{m}^{1/2}}\exp\left(\sum_{\substack{p|m}} \sum_{j=1}^\infty \frac{p^{-j/2}}{j}\right) 
 \\& \leq  x^{1/2}\exp\left(\sum_{\substack{p\leq 2\log x}}\left( p^{-1/2} + \frac{1}{2p} + \sum_{j=3}^\infty \frac{p^{-j/2}}{j}\right)\right)\sum_{\substack{d\leq x^2 \\ d \text{ squarefull}}} \frac{1}{d^{1/2}}.
\end{align*}
By the prime number theorem we have \[\sum_{\substack{p\leq 2\log x}} p^{-1/2} = \text{li}\left(\left(2\log x\right)^{1/2}\right)\left(1+O\left(\frac{1}{\log\log x}\right)\right).\]
So we can rewrite the expression above as
 \begin{align*}
 x^{1/2}\exp & \left( \text{li} \left((2\log x)^{1/2}\right)\left(1+O\left(\tfrac{1}{\log\log x}\right)\right  ) + \tfrac{1}{2}\log\log\log x  + O(1) \right) \hspace{-4mm}\sum_{\substack{d\leq x^2 \\ d \text{ squarefull}}}\hspace{-4mm}\frac{1}{d^{1/2}} \\
 &=  x^{1/2}\exp\left(\frac{2(2\log x)^{1/2}}{\log\log x}\left(1+O\left(\tfrac{1}{\log\log x}\right)\right)\right) \hspace{-4mm}\sum_{\substack{d\leq x^2 \\ d \text{ squarefull}}}\hspace{-4mm}\frac{1}{d^{1/2}}.
\end{align*}

By partial summation, we see that \[\sum_{\substack{d\leq x^2 \\ d \text{ squarefull}}}\hspace{-4mm}\frac{1}{d^{1/2}} = O(\log x),\]
 which can be absorbed into the existing error term in our equation, proving the theorem.
 \end{proof}
 
Note that if we assume a strong form of the prime $k$-tuples conjecture, due to Hardy and
 Littlewood, we can show that this is fairly close to the actual size of $K_2(x)$.  Their conjecture implies that the number of integers $m$ up to $x^{1/2}$ with both $m+1$ and
 $2m+1$ prime is asymptotically $cx^{1/2}/(\log x)^{2}$.  Now, whenever both are prime, (and $m \neq 1$) we see that
$\kappa((m+1)(2m+1)) = \rad{2m^2} = \rad{m}$, (since $m$ is necessarily even) and ${\rad{m}|(m+1)(2m+1)-1}$. Thus $K_2(x)$ would be at least of order $x^{1/2}/(\log x)^{2}$.
 
 \section{$k$-Lehmer Numbers}
 Grau and Oller-Marc{\'e}n \cite{Grau} define a $k$-Lehmer number to be an integer $n$ satisfying the condition $\varphi(n)|(n-1)^k$.  (Note that they do not require $n$ to be composite, as we have in our definitions.)  In their paper they make several conjectures about the counts of these $k$-Lehmer numbers.  Our Theorem \ref{Main}, which shows in particular that $K(x) = O(\pi(x))$ (where $\pi(x)$ is the prime counting function) resolves four of these conjectures, Conjectures 8 (i)-(iv).  Namely, this result proves Conjectures 8 (i),(ii) and (iv), while disproving (iii).  Our methods, combined with the methods used in \cite{Lehmer23} to obtain a bound on the Lehmer numbers, can also be used to bound the counts of the $k$-Lehmer numbers.
 
We let $\mathbb{L}_k(x)$ be the set of composite $n$ up to $x$ which satisfy ${\varphi(n)|(n-1)^k}$, and $L_k(x) = |\mathbb{L}_k(x)|$.  (So Grau and Oller-Marc{\'e}n's function $C_k(x)=L_k(x)+\pi(x) +1$.)
  
\begin{theorem} For $k \geq 2$ we have $L_k(x)\ll_k x^{1- \frac{1}{4k-1}}$.
\end{theorem}
\begin{proof} We consider three cases, based on the size of the largest prime divisor.  We consider first those $n$, $x^{1-\frac{1}{4k-1}}< n \leq x$, which have $P(n)<x^{\frac{k}{4k-1}}$.  Any such $n$ will have a divisor $d$ in the range $(x^\frac{k}{4k-1},x^\frac{2k}{4k-1})$. Write $n = md$, so $m \leq /d$ and since $\varphi(md)|(md-1)^k$, we see that $(md-1)^{k} \equiv 0 \pmod{\varphi(d)}$.  
  
Now, for any positive integer $N$, the number of residue classes $r \pmod{N}$ with $r^k \equiv 0 \pmod{N}$ is at most $N^\frac{k-1}{k}$.  Thus, for any fixed $d$, using the fact that $(d,\varphi(d))=1$, we see that $m$ must be in one of at most $\varphi(d)^{\frac{k-1}{k}}$ residue classes mod $\varphi(d)$, giving us at most \[\varphi(d)^{\frac{k-1}{k}}\left\lceil\frac{x}{d\varphi(d)}\right\rceil \leq \varphi(d)^{\frac{k-1}{k}}\left(1+\frac{x}{d\varphi(d)}\right)\]
choices for $m$.
  
  Summing over all $d$ in the range $I = (x^\frac{k}{4k-1},x^\frac{2k}{4k-1})$, we get  
  
\begin{align*}
  \sum_{d \in I} \varphi(d)^{\frac{k-1}{k}}\left(1+\frac{x}{d\varphi(d)}\right) &\leq \sum_{d \in I} d^\frac{k-1}{k} + \frac{x}{d^{1+\frac{1}{k}}}\left(\frac{d}{\varphi(d)}\right)^{\frac{1}{k}} \\
& \leq \sum_{d \in I} d^\frac{k-1}{k} + \sum_{d \in I} \frac{x}{d^{1+\frac{1}{k}}}\left(\frac{d}{\varphi(d)}\right). \\
\end{align*}
The first sum is $\ll x^{1-\frac{1}{4k-1}}$. 
Now, using partial summation on the second sum and the fact that $\sum_{t \leq x} \frac{t}{\varphi(t)} = O(x)$, we get
\begin{align*}
\sum_{d \in I} \frac{x}{d^{1+\frac{1}{k}}}\left(\frac{d}{\varphi(d)}\right)  &\ll \frac{x}{x^{(\frac{2k}{4k-1})({1+\frac{1}{k}})}}\sum_{d\leq  x^\frac{2k}{4k-1}}\frac{d}{\varphi(d)}+  x\int_{x^\frac{k}{4k-1}}^{x^\frac{2k}{4k-1}}\frac{1}{t^{2+\frac{1}{k}}}\sum_{i<t}\frac{t}{\varphi(t)}dt\\
&\ll \frac{x}{x^{(\frac{2k}{4k-1})({1+\frac{1}{k}})}}\left(x^\frac{2k}{4k-1}\right) +  x\int_{x^\frac{k}{4k-1}}^{x^\frac{2k}{4k-1}}\frac{1}{t^{1+\frac{1}{k}}}dt\\
&\ll_k x^{1-\frac{2}{4k-1}} + \frac{x}{x^{(\frac{k}{4k-1})({\frac{1}{k}})}} \ll x^{1-\frac{1}{4k-1}}.
\end{align*}
  
In the second case we consider those $n$ with $x^{\frac{k}{4k-1}}<P(n) \leq x^{\frac{2k}{4k-1}}$. In this case $n$ again has a divisor in the range $(x^\frac{k}{4k-1},x^\frac{2k}{4k-1})$, namely $p$, and the above argument applies verbatim.
  
Finally we've reduced to the case that $P(n)>x^{\frac{2k}{4k-1}}$, and the argument used for large primes in our main theorem gives us that the number of $n$ with $\kappa(n)|n-1$ and $P(n)> x^{\frac{2k}{4k-1}}$ is at most $x^{1-\frac{k}{4k-1}}$, hence for those $n$ in $\mathbb{L}_k(x)$ as well, and our result follows.
\end{proof} 
  We note that it may be possible to improve upon this bound by using  techniques developed in more recent papers to obtain better bounds on the Lehmer numbers.
  
\section{Computations and Conjectures}
 \label{sec:computations}
 Table \ref{table} shows the values of $K(x)$ we computed for increasing powers of 10, compared with values of $C(x)$, computed by Richard Pinch \cite{Pinch}.  Our computations were done using trial divison, in which a candidate number, $n$, was rejected as soon as soon as it was found to be nonsquarefree, or to have a prime divisor $p$, which failed to satisfy $\rad{p-1}|n-1$.
 \begin{table}[ht]
 \caption{Values of $C(x)$ and $K(x)$ to $10^{11}$.}
 \label{table}
    \begin{tabular}{|l|l|l|}
        \hline
        $n$ & $C(10^n)$ & $K(10^n)$ \\ \hline 
        2   & 0         & 4         \\ 
        3   & 1         & 19        \\ 
        4   & 7         & 103       \\ 
        5   & 16        & 422       \\ 
        6   & 43        & 1559      \\ 
        7   & 105       & 5645      \\ 
        8   & 255       & 19329     \\ 
        9   & 646       & 64040     \\ 
        10  & 1547      & 205355    \\ 
        11  & 3605      & 631949         \\
        \hline
    \end{tabular}
    \end{table}
    
Despite the similar asymptotic bounds that we have for $C(x)$ and $K(x)$, it is clear that $K(x)$ is growing substantially faster, which leads to the conjecture:
\begin{conjecture} $\lim_{x \to \infty} K(x)/C(x) = \infty$. 
\end{conjecture}
    
At the moment, however, we are unable to prove even the much weaker conjecture:  
\begin{conjecture} $\lim_{x \to \infty} K(x) - C(x) = \infty$.
\end{conjecture}
\section*{Acknowledgments}
I would like to thank my advisor, Carl Pomerance, for suggesting the problem and for his invaluable guidance and encouragement throughout the development of this paper.
\bibliographystyle{amsplain} 
\bibliography{radbib}
\end{document}